\documentclass[12pt]{amsart}
\usepackage{amsmath,amssymb,amsbsy,amsfonts,latexsym,amsopn,amstext,cite,
                                               amsxtra,euscript,amscd,bm}
\usepackage{url}

\usepackage{mathrsfs}

\usepackage{color}
\usepackage[colorlinks,linkcolor=blue,anchorcolor=blue,citecolor=blue,backref=page]{hyperref}
\usepackage{color}
\usepackage{graphics,epsfig}
\usepackage{graphicx}
\usepackage{float}
\usepackage{epstopdf}
\hypersetup{breaklinks=true}

\usepackage[np]{numprint}
\npdecimalsign{\ensuremath{.}}

\def\Th1{\varTheta}

\def\OmQ{\Omega^\Q}

\usepackage{bibentry}

\usepackage[english]{babel}
\usepackage{mathtools}
\usepackage{todonotes}
\usepackage{url}
\usepackage[colorlinks,linkcolor=blue,anchorcolor=blue,citecolor=blue,backref=page]{hyperref}

\usepackage[norefs,nocites]{refcheck}

\begin{document}

\newtheorem{theorem}{Theorem}
\newtheorem{lemma}[theorem]{Lemma}
\newtheorem{claim}[theorem]{Claim}
\newtheorem{cor}[theorem]{Corollary}
\newtheorem{conj}[theorem]{Conjecture}
\newtheorem{prop}[theorem]{Proposition}
\newtheorem{definition}[theorem]{Definition}
\newtheorem{question}[theorem]{Question}
\newtheorem{example}[theorem]{Example}
\newcommand{\hh}{{{\mathrm h}}}
\newtheorem{remark}[theorem]{Remark}

\numberwithin{equation}{section}
\numberwithin{theorem}{section}
\numberwithin{table}{section}
\numberwithin{figure}{section}

\def\sssum{\mathop{\sum\!\sum\!\sum}}
\def\ssum{\mathop{\sum\ldots \sum}}
\def\iint{\mathop{\int\ldots \int}}

\newcommand{\diam}{\operatorname{diam}}

\def\squareforqed{\hbox{\rlap{$\sqcap$}$\sqcup$}}
\def\qed{\ifmmode\squareforqed\else{\unskip\nobreak\hfil
\penalty50\hskip1em \nobreak\hfil\squareforqed
\parfillskip=0pt\finalhyphendemerits=0\endgraf}\fi}

\newfont{\teneufm}{eufm10}
\newfont{\seveneufm}{eufm7}
\newfont{\fiveeufm}{eufm5}
%
%
\newfam\eufmfam
     \textfont\eufmfam=\teneufm
\scriptfont\eufmfam=\seveneufm
     \scriptscriptfont\eufmfam=\fiveeufm
%
%
\def\frak#1{{\fam\eufmfam\relax#1}}

\newcommand{\bflambda}{{\boldsymbol{\lambda}}}
\newcommand{\bfmu}{{\boldsymbol{\mu}}}
\newcommand{\bfxi}{{\boldsymbol{\eta}}}
\newcommand{\bfrho}{{\boldsymbol{\rho}}}

\def\eps{\varepsilon}

\def\fK{\mathfrak K}
\def\fT{\mathfrak{T}}
\def\fL{\mathfrak L}
\def\fR{\mathfrak R}

\def\fA{{\mathfrak A}}
\def\fB{{\mathfrak B}}
\def\fC{{\mathfrak C}}
\def\fM{{\mathfrak M}}
\def\fS{{\mathfrak  S}}
\def\fU{{\mathfrak U}}
\def\fW{{\mathfrak W}}

\def\T {\mathsf {T}}
\def\Tor{\mathsf{T}_d}
\def\Tore{\widetilde{\mathrm{T}}_{d} }

\def\sM {\mathsf {M}}

\def\ss{\mathsf {s}}

\def\Kmnd{\cK_d(m,n)}
\def\Kmnp{\cK_p(m,n)}
\def\Kmnq{\cK_q(m,n)}

\def \balpha{\bm{\alpha}}
\def \bbeta{\bm{\beta}}
\def \bgamma{\bm{\gamma}}
\def \bdelta{\bm{\delta}}
\def \bzeta{\bm{\zeta}}
\def \blambda{\bm{\lambda}}
\def \bchi{\bm{\chi}}
\def \bphi{\bm{\varphi}}
\def \bpsi{\bm{\psi}}
\def \bnu{\bm{\nu}}
\def \bomega{\bm{\omega}}

\def \bell{\bm{\ell}}

\def\eqref#1{(\ref{#1})}

\def\vec#1{\mathbf{#1}}

\newcommand{\abs}[1]{\left| #1 \right|}

\def\Zq{\mathbb{Z}_q}
\def\Zqx{\mathbb{Z}_q^*}
\def\Zd{\mathbb{Z}_d}
\def\Zdx{\mathbb{Z}_d^*}
\def\Zf{\mathbb{Z}_f}
\def\Zfx{\mathbb{Z}_f^*}
\def\Zp{\mathbb{Z}_p}
\def\Zpx{\mathbb{Z}_p^*}
\def\cM{\mathcal M}
\def\cE{\mathcal E}
\def\cH{\mathcal H}

\def\le{\leqslant}

\def\ge{\geqslant}

\def\sfB{\mathsf {B}}
\def\sfC{\mathsf {C}}
\def\L{\mathsf {L}}
\def\FF{\mathsf {F}}

\def\sE {\mathscr{E}}
\def\sS {\mathscr{S}}

\def\cA{{\mathcal A}}
\def\cB{{\mathcal B}}
\def\cC{{\mathcal C}}
\def\cD{{\mathcal D}}
\def\cE{{\mathcal E}}
\def\cF{{\mathcal F}}
\def\cG{{\mathcal G}}
\def\cH{{\mathcal H}}
\def\cI{{\mathcal I}}
\def\cJ{{\mathcal J}}
\def\cK{{\mathcal K}}
\def\cL{{\mathcal L}}
\def\cM{{\mathcal M}}
\def\cN{{\mathcal N}}
\def\cO{{\mathcal O}}
\def\cP{{\mathcal P}}
\def\cQ{{\mathcal Q}}
\def\cR{{\mathcal R}}
\def\cS{{\mathcal S}}
\def\cT{{\mathcal T}}
\def\cU{{\mathcal U}}
\def\cV{{\mathcal V}}
\def\cW{{\mathcal W}}
\def\cX{{\mathcal X}}
\def\cY{{\mathcal Y}}
\def\cZ{{\mathcal Z}}
\newcommand{\rmod}[1]{\: \mbox{mod} \: #1}

\def\cg{{\mathcal g}}

\def\vy{\mathbf y}
\def\vr{\mathbf r}
\def\vx{\mathbf x}
\def\va{\mathbf a}
\def\vb{\mathbf b}
\def\vc{\mathbf c}
\def\ve{\mathbf e}
\def\vh{\mathbf h}
\def\vk{\mathbf k}
\def\vm{\mathbf m}
\def\vz{\mathbf z}
\def\vu{\mathbf u}
\def\vv{\mathbf v}

\def\e{{\mathbf{\,e}}}
\def\ep{{\mathbf{\,e}}_p}
\def\eq{{\mathbf{\,e}}_q}

\def\Tr{{\mathrm{Tr}}}
\def\Nm{{\mathrm{Nm}}}

 \def\SS{{\mathbf{S}}}

\def\lcm{{\mathrm{lcm}}}

 \def\0{{\mathbf{0}}}

\def\({\left(}
\def\){\right)}
\def\l|{\left|}
\def\r|{\right|}
\def\fl#1{\left\lfloor#1\right\rfloor}
\def\rf#1{\left\lceil#1\right\rceil}
\def\sumstar#1{\mathop{\sum\vphantom|^{\!\!*}\,}_{#1}}

\def\mand{\qquad \mbox{and} \qquad}

\def\tblue#1{\begin{color}{blue}{{#1}}\end{color}}




\hyphenation{re-pub-lished}

\mathsurround=1pt

\def\bfdefault{b}

\def \F{{\mathbb F}}
\def \K{{\mathbb K}}
\def \N{{\mathbb N}}
\def \Z{{\mathbb Z}}
\def \P{{\mathbb P}}
\def \Q{{\mathbb Q}}
\def \R{{\mathbb R}}
\def \C{{\mathbb C}}
\def\Fp{\F_p}
\def \fp{\Fp^*}

 \def \xbar{\overline x}

\title[Two-parametric Weyl sums along smooth curves]{Hybrid bounds on two-parametric family Weyl sums along smooth curves} 


 \author[C. Chen] {Changhao Chen}

\address{Department of Mathematics, The Chinese University of Hong Kong, Shatin, Hong
Kong}
\email{changhao.chenm@gmail.com}

 \author[I. E. Shparlinski]{Igor E. Shparlinski}

\address{Department of Pure Mathematics, University of New South Wales,
Sydney, NSW 2052, Australia}
\email{igor.shparlinski@unsw.edu.au}

\begin{abstract}  
We  obtain a new bound  on   Weyl sums with degree $k\ge 2$ polynomials of the form
$(\tau x+c) \omega(n)+xn$, $n=1, 2, \ldots$,  with fixed  $\omega(T) \in \Z[T]$ and $\tau \in \R$,  
which holds for almost all $c\in [0,1)$ and all $x\in [0,1)$. We improve and generalise some recent 
results of M.~B.~Erdo{\v g}an and G.~Shakan (2019), whose work also shows links between 
this question and some classical partial differential equations. We extend  this to 
more general settings of families of polynomials $xn+y \omega(n)$ for  all $(x,y)\in [0,1)^2$ with $f(x,y)=z$ 
for a set of $z \in [0,1)$ of  full  Lebesgue measure, provided that $f$ is some  H\"older function.   
\end{abstract}


\maketitle

\tableofcontents

\section{Introduction}

\subsection{Background}    For a natural number  $d $  let $\Tor = (\R/\Z)^d$ be the  $d$-dimensional unit torus. 
We also write $\T= \R/\Z$ instead of $\T_1$. 

Given a family $\bphi = \(\varphi_1(T), \ldots, \varphi_d(T)\)\in \Z[T]^d$  of $d$ distinct 
nonconstant polynomials and a vector $\vu=(u_1, \ldots, u_d)\in \Tor$, we  consider  the Weyl~\cite{Weyl} sums
 \begin{equation}
\label{eq:SSu} 
\cS_{\bphi}( \vu; N)=\sum_{n=1}^{N}  \e\(u_1 \varphi_1(n)+\ldots + u_d\varphi_d(n) \),
\end{equation}
where throughout the paper we denote 
$$
\e(x) = \exp(2\pi ix).
$$

Recently, Wooley~\cite{Wool3} (see also Flaminio and  Forni~\cite{FlFo}) has  introduced a  scenario  which interpolates between 
{\it individual\/}  bounds and {\it bounds\/} involving  averaging over all $\vu \in\Tor$.
In the setting of~\cite{Wool3}    the sums $\cS_{\bphi}( \vu; N)$ are estimated 
for almost all (with respect to the  Lebesgue measure) coordinates $u_i$, $i \in \cI$, and 
for all  coordinates $u_j$, $j \in \cJ$,  where the sets $\cI$ and $\cJ$ form a partition of the set 
$\{1, \ldots, d\}$.  The results of Wooley~\cite{Wool3}  have been recently improved and generalised in~\cite{ChSh-IMRN}.  
To be precise, we  outline a special version  of Wooley~\cite{Wool3}, Flaminio and  Forni~\cite{FlFo} and
the authors~\cite{ChSh-IMRN} with  $d=2$. Let $\varphi_1(T), \varphi_2(T)\in \Z(T)$ then there exists a positive 
constant $\rho<1$ depending only on $\varphi_1$ and $\varphi_2$ such that for almost all $u\in \T$ we have 
\begin{equation}
\label{eq:type}
\sup_{v\in \T} \left| \sum_{n=1}^{N}\e(u\varphi_1(n)+v\varphi_2(n)) \right|\le N^{\rho+o(1)}, \quad N\rightarrow \infty.
\end{equation}

Independently, motivated by applications to some families of partial differential equations, 
Erdo{\v g}an  and Shakan~\cite{ErdSha}  (see also~\cite{BPPSV}) have considered the following more special
 case of dimension $d=2$ with   $\bphi = (\omega(T),  \tau \omega(T) + T)$ 
 for some $\tau \in \Q$ and  function $\omega: \Z \to \R$ (not necessary a polynomial). We now present some details 
 for the motivation of~\cite{BPPSV, ErdSha}.   Exponential series of  the type  
\begin{equation}
\label{eq:q(t,x)}
q(t, x)=\sum_{n\in \Z} a_n \e(t \omega(n)+xn)
\end{equation}
are  solutions of various partial differential equations with respect to different function $\omega: \Z\rightarrow \R$.  
The key examples are the {\it linear Schr\"odinger equation\/}:
$$
i q_t+q_{xx}=0, \quad  \text{with} \ \omega(n)=-n^2,
$$
and the {\it Airy equation\/}: 
$$
q_t+q_{xxx}=0, \quad  \text{with} \ \omega(n)=n^{3}.
$$ 
Hence,  it  is important  to investigate the properties of the exponential  series~\eqref{eq:q(t,x)}; we refer to Erdo{\v g}an  and   Shakan~\cite{ErdSha} for more details. 
Among other things, Erdo{\v g}an  and  Shakan~\cite{ErdSha} have  obtained bounds on  the {\it Minkowski\/}, or {\it box\/},  {\it dimension\/}  of the graphs of real and imaginary parts of the function  
$$
q(t,x ) {\big |} _{t=\tau x+c} = q(\tau x+c,x )  ,
$$
for almost all $c\in \R$ and any fixed rational number $\tau$, where, as usual, $q(t,x)|_{t=\tau x+c}$ means that we consider (or restrict) the function $q(t, x)$ on the line $t=\tau x+c$. 
For this purpose, Erdo{\v g}an  and  Shakan~\cite{ErdSha}  obtain exponential sum estimates of the following type: there exists a positive constant $\vartheta<1 $ depending only on $\omega$ such that for any $\tau\in \Q$ and almost all $c\in \R$ the following estimates    
 \begin{equation}
\label{eq:desired bound} 
\sup_{x\in \T}\left | \sum_{n=1}^N   \e((\tau x+c) \omega(n)+xn) \right| \le  N^{\vartheta+o(1)},
\end{equation}
holds as $N \to \infty$.
Note that this corresponds to sums~\eqref{eq:type}  with $\bphi = (\omega(T),  \tau \omega(T) + T)$. 
It is important to remark that the uniformity with respect to $c$ and $\tau$ is {\it not required\/} in~\eqref{eq:desired bound}.

Assume that for some $\vartheta<1$,  for any $\tau\in \Q$, for a set of $c \in \T$ of full  Lebesgue measure we have~\eqref{eq:desired bound}
(again, the uniformity with respect to $c$ and $\tau$ is  not required).   
Then, for polynomials $\omega$ with $\deg \omega = k \ge 2$, 
 the argument of the proof of~\cite[Corollary~3.5]{ErdSha} under the assumption~\eqref{eq:desired bound} with any $\vartheta < 1$
 gives  a nontrivial bound 
$$
\delta \le 2 -  (1-\vartheta)/k
$$ 
on the {\it fractal dimension\/} $\delta$  of the graph of the Fourier coefficients of  solutions to some 
linear dispersive partial differential equation, see~\cite[Equation~(1)]{ErdSha}. We refer for further 
details to~\cite{ErdSha}, see also~\cite{BPPSV, ChCo, ChCoUb, ChUb, Pierce} for some related questions 
and further references. 

Here we concentrate on obtaining new bounds of the form~\eqref{eq:desired bound}  for polynomials $\omega(T) \in\Z[T]$
and in particular, using several ideas from~\cite{ChSh-IMRN} we improve  and generalise some bounds from~\cite{ErdSha}.

\subsection{Previous results} 

We say that a certain property holds for {\it almost all $u \in \T$} if it holds for a set $\cU \subseteq \T$ 
of {\it Lebesgue measure\/} $\lambda(\cU) = 1$. 

We define $\OmQ_k$ as the smallest possible value (infimum) of $\vartheta$ such that  for any polynomials   $\omega(T) \in \Z[T]$ of degree $k$ and any $\tau\in \Q$ there is a set of $u\in \T$ of full Lebesgue measure satisfying~\eqref{eq:desired bound}.



We also define $\varTheta_k$ as the smallest possible value (infimum) of $\vartheta$ such that  for the  polynomial   
$\omega(T) = T^k$ and  $\tau =1$  there is a set of $u\in \T$ of full Lebesgue measure satisfying~\eqref{eq:desired bound}.

%

The goal is to improve the trivial bounds
$$
 \OmQ_k \le 1 \quad {and} \quad   \varTheta_k\le 1.
$$

We remarks that in~\cite{ErdSha} (as well as in~\cite{BPPSV}) the Weyl sums in~\eqref{eq:desired bound}  are 
over dyadic intervals, but both formulations are certainly equivalent.   

The first nontrivial bound  
\begin{equation}
\label{eq:Omegak-1}
\OmQ_k \le  1-\frac{1}{2^k +1}
\end{equation} 
has been given by  Erdo{\v g}an  and Shakan~\cite[Proposition~3.3]{ErdSha}, which is based on 
the classical {\it Weyl differencing method\/}, see~\cite[Lemma~2.4]{Vau}. 
It is also noticed in~\cite[Footnote~6]{ErdSha} that for large $k$  
 the bound~\eqref{eq:Omegak-1} can be improved if one uses
the Vinogradov method to bound Weyl sums combined with the modern form of the the Vinogradov
due to Bourgain, Demeter and Guth~\cite{BDG} (for $d \geqslant 4$) and Wooley~\cite{Wool2} (for $d=3$), 
see~\eqref{eq:MVT} below. Namely one easily verifies that using this bound, 
see, for example,~\cite[Theorem~5]{Bourg} or~\cite[Equation~(16)]{ErdSha}, one derives 
\begin{equation}
\label{eq:Omegak-2}
\OmQ_k\le  1-\frac{1}{2k(k-1) +1}.
\end{equation} 
Here we improve this bound and extend it to $\OmQ_k$  as follows following, 
\begin{equation}
\label{eq:all-tau}
\Omega_k^{\R}\le  1-\frac{1}{2s_0(k)+1},
\end{equation} 
where $s_0(k)$ is given at~\eqref{eq:s0 small k} and~\eqref{eq:s0 large k} below. In fact the bound~\eqref{eq:all-tau}
is a very special case of a much more general result given in Theorem~\ref{thm:general} below.

In the monomial case, the truly remarkable result of~\cite[Theorem~1.4]{BPPSV} gives {\it exact} values
\begin{equation}
\label{eq:Theta23}
 \varTheta_2   =   \varTheta_3  = \frac{3}{4}.
\end{equation} 

It is  very   interesting that the exact values of $ \varTheta_2 $ and $\varTheta_3 $ in~\eqref{eq:Theta23} differ from the naively 
expected $1/2$.

\subsection{Set-up}

For $\omega(T) \in \Z[T]$ and $(x, y)\in \T_2$   we  consider the two-parametric  family of Weyl sums
 \begin{equation}
\label{eq:Suv} 
S_{\omega}(x, y; N)=\sum_{n=1}^{N} \e\(xn+y \omega(n)  \).
\end{equation}  
For $\vu=(x, y)\in \T_2$ we also use the notation
$$
S_{\omega}(\vu; N) = S_{\omega}(x, y; N)
$$
as well, and we also apply this convention to other similar sums.

We note that informally~\eqref{eq:desired bound} means the existence of a nontrivial bound of Weyl sums $S_{\omega}(x, y; N)$ 
along the all points bundle of lines $y =\tau x + c$ which holds for any $\tau \in \Q$ and almost all $c \in \R$. 

Certainly besides relaxing the condition $\tau \in \Q$  to $\tau \in \R$  it is also interesting to extend the above family of straight lines 
to more general curves, satisfying some smoothness conditions.  

Let  $0<\rho\le 1$.
We recall that a function  $f:\T_2\rightarrow \R$  is called  a {\it $\rho$-H\"older function\/}, if there is some constant $C(f)$ depending only on $f$ 
such thar 
$$
\|f(\vu)-f(\vv)\|\le C(f) \|\vu-\vv\|^\rho,  \qquad \vu, \vv \in \T_2, 
$$
where $\| \vec{z}\|$ is the Euclidean norm of $\vec{z}$ (note that the left side of this inequality is the Euclidean norm in $\R$, while the right side is the Euclidean norm in $\R^2$).

In particular, in the case $\rho=1$ the function $f$ is often called a {\it Lipschitz function\/}. Moreover note that if $f$ is a differentiable  function on $\T_2$ 
and the partial differentials $f_x, f_y$ are uniformly bounded then the function is a Lipschitz function.

Furthermore,  for  a function $f: \T_2\rightarrow \R$  and   $z\in \R$  we  denote  the {\it level set}
$$
f^{-1}(z)=\{(x, y):~f(x, y)=z\}.
$$

Now for an integer $k\ge 2$ and $0<\rho\le 1$.  we denote $\Omega_{k, \rho}$ as the smallest possible value (infimum) of $\vartheta$ 
such that for any $\omega(T)\in \Z[T]$ of degree $k$, any $\rho$-H\"older function  $f:\T_2\rightarrow \R$  
and for almost all $z\in R(f)$ we have 
$$
\sup_{(x, y)\in  f^{-1}(z)}\ \left|S_{ \omega}(x, y; N)\right| \le N^{\vartheta + o(1)}.
$$
We again note that the uniformity in $f$ is not required. 

Since we obviously have 
$$
\Omega_k^{\R} \le  \Omega_{k, 1},
$$
taking  $\rho=1$ in Theorem~\ref{thm:general} below we obtain the aforementioned improvement~\eqref{eq:all-tau}
of~\eqref{eq:Omegak-1} and~\eqref{eq:Omegak-2} from~\cite{ErdSha}.

\subsection{New results} 
Following the definition of $\vartheta(k)$ of Wooley~\cite[Equation~(14.22)]{Wool5}, 
it is convenient to  introduce the following quantity  
\begin{equation}
\label{eq:etak} 
\eta(k)= \begin{cases} 0, & \quad 2k+2 \ge \fl{\sqrt{2k+2}}^2 + \fl{\sqrt{2k+2}},\\
1, & \quad \text{otherwise.} 
\end{cases}
\end{equation}

We now define $s_0$  as follows. For $k \in \{2, \ldots, 10\}$ we set   
\begin{equation}
\begin{split}
\label{eq:s0 small k}
&s_0(2) = 3, \qquad  s_0(3) = 5,\qquad  s_0(4) = 8,\\ 
&s_0(5) = 12 ,\qquad   s_0(6) = 18,
 \qquad s_0(7) = 24, \\ &s_0(8) =31, \qquad  s_0(9) = 40,\qquad  s_0(10) = 49,
\end{split}
\end{equation}
while for $k \ge 11$ 
 we define 
\begin{equation} 
\label{eq:s0 large k}
s_0(k) = k(k-1)/2 + \fl{\sqrt{2k+2}}  - \eta(k) .
\end{equation}

\begin{theorem}
\label{thm:general} 
For any integer $k\ge 2$ and  $0<\rho\le 1$., we have 
$$
 \Omega_{k, \rho } \le  1-\frac{\rho}{2s_0(k) +2-\rho}.
$$
\end{theorem}

Taking  $\rho=1$, we see that  Theorem~\ref{thm:general} yields~\eqref{eq:all-tau} and shows that the same  bound holds along almost  curves define by level sets of a  Lipschitz function.

Moreover, let $f(x, y)=x^2+y^2$ then $f$ is a Lipschitz function, and hence Theorem~\ref{thm:general} implies that for almost all $0<z<1$ we have 
$$
\sup_{(x, y)\in f^{-1}(z)}\ \left|S_{ \omega}(x, y; N)\right| \le N^{ 1- 1/(2s_0(k) +1) + o(1)}, \quad N\rightarrow \infty.
$$
Note that  the supremum is taken over a family of circles $x^2+y^2 = z$.

We now show some possible improvements and variants of Theorem~\ref{thm:general} for some special functions $f$.
 We start with the {\it  projection\/}  $f(x, y)=y$ (also corresponding to $\tau = 0$  in~\eqref{eq:desired bound}.
 We denote by $\Pi_k$ the analogue of  $\Omega_{k, \rho}$ to only one function  $f(x, y)=y$.

\begin{theorem} 
\label{thm:v} 
For any integer $k\ge 2$ we have 
$$
\Pi_k \le  1-\frac{k}{2s_0(k) +1}.
$$
\end{theorem}

Our methods  also yield the result on restricting Weyl sums on a larger family of circles. More precisely, for $0<r<1$ and  $\vz\in \T_2$, denote  
$$
\cC(\vz, r) = \{\vu \in \R^2:~\|\vu -\vz\| = r\}
$$  
the circle with center $\vz$ and radius $r$.

We now define $\Gamma_k$   as the smallest possible value (infimum) of $\vartheta$ such that  for any polynomials   
$\omega(T) \in \Z[T]$ of degree $k$ and any $r\in (0,1)$  there is a set of  $(x, y)\in \T_2$   of full Lebesgue
such that 
$$
\sup_{(x,y)\in \cC(\vz, r) } |S_{\omega}(x,y; N)|\le N^{1-\vartheta+o(1)}
$$
as $N \to \infty$. 

\begin{theorem}\label{thm:c}
For any integer $k\ge 2$ we have 
$$
\Gamma_k \le 1-\frac{1}{2s_0(k) +1}.
$$
\end{theorem}

Note that the bound in Theorem~\ref{thm:c} is the same bound as Theorem~\ref{thm:general} for the case $\rho=1$. We may expect that Theorem~\ref{thm:c} follows from Theorem~\ref{thm:general}. However, it seems that the argument is not 
immediately obvious. As for Theorem~\ref{thm:general}, the claim holds for almost all $z\in \T$ with respect to one dimensional Lebesgue measure. While in Theorem~\ref{thm:c} the claim holds for almost all $\vz=(z_1, z_2)\in \T_2$ with respect to the two dimensional Lebesgue measure. While this can be handled via the {\it Fubini theorem\/} there are still some issues
with the uniformity of constants in the Lipschitz condition on the relevant functions.

\section{Preparations}

\subsection{Notation and conventions}

Throughout the paper, the notation $U = O(V)$, 
$U \ll V$ and $ V\gg U$  are equivalent to $|U|\leqslant c V$ for some positive constant $c$, 
which throughout the paper may depend, where obvious, on the polynomial $\omega$ 
(or sometimes only on $k = \deg \omega$) and the real $\tau$
and  are absolute otherwise.

For any quantity $V> 1$ we write $U = V^{o(1)}$ (as $V \to \infty$) to indicate a function of $V$ which 
satisfies $|U| \le V^\eps$ for any $\eps> 0$, provided $V$ is large enough. The advantage 
of using $V^{o(1)}$ is that it absorbs $\log V$ and other similar quantities without changing  the whole 
expression and the need to re-define $\eps$.

\subsection{Mean value theorems}
\label{sec:MVT}

We start with recalling the  the Vinogradov mean value theorem for  the Weyl sums has recently been 
established by Bourgain, Demeter and Guth~\cite{BDG} (for $d \geqslant 4$) 
and Wooley~\cite{Wool2} (for $d=3$)  (see also~\cite{Wool5}) in the best possible 
form  
\begin{equation}
\label{eq:MVT}
\int_{\T_k}\left|\sum_{n=1}^N \e\(x_1n+\ldots+x_k n^k\)\right|^{2s}\, d\vx \le   N^{s+o(1)} + N^{2s - s(k)+o(1)} 
\end{equation}
as $N \to  \infty$ and the integration is over $\vx = (x_1, \ldots, x_k) \in \T_k$ 
and
$$
s(k) = \frac{k(k+1)}{2}.
$$

However, for our application we use a special case of  a result of Wooley~\cite[Corollary~14.8]{Wool5} and combined with 
the table at the end of~\cite[Section~14]{Wool5} for $k\ge 4$, and the classical method of Hua~\cite{Hua}, see also~\cite[Lemma~5]{BrRob}
or~\cite[Equation~(14.27)]{Wool5} for $k=2,3$, see also~\cite[Table~1]{ACHK}.

We define $\sigma_0(k)$ as follows. For $k \in \{2, \ldots, 10\}$ we set 
\begin{equation}
\begin{split}
\label{eq:small k}
&\sigma_0(2) = 6, \qquad   \sigma_0(3) = 10,\qquad   \sigma_0(4) = 15,\\ 
& \sigma_0(5) = 70/3,\qquad    \sigma_0(6) = 34,
 \qquad \sigma_0(7) = 93/2, \\ &\sigma_0(8) = 306/5, \qquad   \sigma_0(9) = 78,\qquad   \sigma_0(10) = 678/7,
\end{split}
\end{equation}
while for $k \ge 11$ 
 we define 
\begin{equation}
\label{eq:large k}
\sigma_0(k) = k(k-1) + 2\fl{\sqrt{2k+2}} -1  - \eta(k) , 
\end{equation} 
where $\eta(k)$ is given by~\eqref{eq:etak}. 

\begin{remark}
We note that for small values of $k$, the underlying result
$$
\int_{0}^1 \int_0^1  \left|\sum_{n=1}^N \e\(xn+ y\omega(n)\)\right|^{2^k+2} \, dx\, dy   \le   N^{2^k-k +1 +o(1)}
$$ 
 within the method  of Hua~\cite{Hua} is traditionally formulated for 
monomials $\omega(T) = T^k$. However one verifies that this bound  holds for any polynomial $\omega(T) \in \Z[T]$ with $\deg \omega = k$.
\end{remark}

\begin{lemma}
\label{lem:Wool Sec14} 
For any polynomial $\omega(T) \in \Z[T]$ of degree $\deg \omega = k \ge 2$, 
and any fixed real $\sigma$ with 
$$
\sigma  \ge   \sigma_0(k), 
$$ 
for $k =2, 3$ or $k\ge 11$ 
and 
$$
\sigma  >   \sigma_0(k), 
$$ 
for integer $k \in [4, 10]$,   
where $\sigma_0(k)$ is given by~\eqref{eq:small k} and~\eqref{eq:large k}, we have 
$$
\int_{0}^1 \int_0^1  \left|\sum_{n=1}^N \e\(xn+ y\omega(n)\)\right|^{\sigma} \, dx\, dy   \le   N^{\sigma -k -1 +o(1)}
$$ 
as $N \to \infty$.
\end{lemma}

\begin{remark}
The term $o(1)$ appears in the exponent of $N$ only in the purely classical  cases $k=2,3$ and in fact can be replaced 
with a power of $\log N$, however  this causes no effect on the final result. 
\end{remark}

For a polynomial $\omega(T) \in \Z[T]$ and an integer $s\ge 1$ we now define the mean value of the exponential polynomials, 
which are  more general than the sums~\eqref{eq:Suv}
$$
I_{ \omega, s}(\va; N) = \int_0^1 \int_0^1 \left|\sum_{n=1}^{N} a_n \e\(xn+y\omega(n)\)\right|^{2s} \, dx\, dy,
$$ 
where $\va = (a_n)_{n=1}^\infty$ is some sequence of complex weights.

\begin{lemma}  
\label{lem:IsN} 
For any polynomial $\omega(T) \in \Z[T]$ of degree $\deg \omega = k\ge 2$, weights $\va$ 
with $a_n = n^{o(1)}$  and fixed integer 
$$
s \ge s_0(k) 
$$ 
where $s_0(k)$ is given by~\eqref{eq:s0 small k} and~\eqref{eq:s0 large k}, we have 
$$
I_{\omega, s}(N) \le   N^{2s  -k -1 +o(1)} 
$$ 
as $N \to \infty$.
\end{lemma}

\begin{proof}  Using that $|z|^2 = z \overline z$ for $z \in \C$ to compute the $2s$-th power 
of the inner sum, after changing the order of  summations and integration, 
we derive 
\begin{align*}
\int_0^1  \int_0^{1}  &\left|\sum_{n =1}^N a_n\e\(xn+y \omega(n)\)\right|^{2s} \, dx\, dy \\
& =  \sum_{n_1, \ldots, n_{2s} =1}^N \overline a_{n_1}  a_{n_2} \ldots  \overline a_{n_{2s-1}}  a_{n_{2s}}\\
&  \qquad \qquad \quad \times \int_0^1   \e\(x f(n_1, \ldots, n_{2s})+y g(n_1, \ldots, n_{2s})\) dxdy,
\end{align*}
where 
$$
f(n_1, \ldots, n_{2s})= \sum_{j=1}^{2s} (-1)^{j}n_j \quad \text{and} \quad g(n_1, \ldots, n_{2s}) =\sum_{j=1}^{2s} (-1)^{j} w(n_j).
$$
Hence,  recalling the condition $a_n = n^{o(1)}$  we conclude that 
 \begin{equation}
\label{eq:I J}
I_{\omega, s}(N)  \le J N^{o(1)},
\end{equation}
where $J$ is the number of solutions to the system of equations
\begin{align*}
f(n_1, \ldots, &n_{2s})=g(n_1, \ldots, n_{2s}) = 0, \\
1  & \le n_1, \ldots, n_{2s} \le N. 
\end{align*} 

Again, using the orthogonality of exponential functions, we  write
$$
J = \int_{0}^1 \int_0^1  \left|\sum_{n=1}^N \e\(xn+ y\omega(n)\)\right|^{2s} \, dx\, dy .
$$ 
One verifies that $\sigma =2 s_0(k)$ is an admissible value of 
$\sigma$ in   Lemma~\ref{lem:Wool Sec14},  which together with~\eqref{eq:I J} concludes the proof. 
\end{proof}

\begin{remark}
\label{rem:non-integr}
Our main ingredient,  Lemma~\ref{lem:Wool Sec14} holds for any real $\sigma\ge \sigma_0(k)$.
However, in Lemma~\ref{lem:IsN}, to pass from $I_{\omega, s}(N) $ to the bound of 
Lemma~\ref{lem:Wool Sec14}  we need the integrality of $s$ (which is not needed for the rest of 
our argument). It is interesting to 
avoid this and find a more efficient way of linking  Lemmas~\ref{lem:Wool Sec14} and~\ref{lem:IsN},
perhaps via an  efficient use of the H{\"o}lder inequality, and thus obtain numerically stronger results 
(which improve with decreasing $s$). 
\end{remark}

\subsection{The completion technique}

We now recall a  result from~\cite{ChSh-IMRN} obtained via the standard   completion technique 
(see~\cite[Section~12.2]{IwKow}) for the Weyl sums $\cS_{\bphi}( \vu; N)$ as in~\eqref{eq:SSu} which we adjust to our  setting of the sums  $S_{\omega}(x, y; N)$ given by~\eqref{eq:Suv}. Namely, by a special case
of~\cite[Lemma~3.2]{ChSh-IMRN} we have

\begin{lemma}\label{lem:control}  
For $(x, y)\in \T_2$ and $1\le M\le N$ we have 
$$
S_{\omega}(x,y ; M) \ll W_{\omega}(x,y; N),
$$
where 
$$
W_{\omega}(x, y; N)=  \sum_{h=-N}^{N} \frac{1}{|h|+1} \left| \sum_{n=1}^{N}   \e\(h n/N+ xn+y \omega(n) \) \right|.
$$
\end{lemma}

\subsection{Continuity of exponential sums}

For $\vu=(x, y) \in \T_2$ and $ \bzeta=(\zeta_1, \zeta_2)$, we define the square centred at $\vu$ with ``side length" $\bzeta$  by  
$$
\cR(\vu; \bzeta)=[x -\zeta_1, x+\zeta_1 )\times  [y-\zeta_2, y+\zeta_2)
$$
 
Here we present  a close analogue of~\cite[Lemma~3.2]{ChSh-IMRN}, 
see also~\cite[Lemma~2.1]{Wool3}. 
However, we use the following version of summation by parts which is slightly different from the 
proof in~\cite[Lemma~3.2]{ChSh-IMRN} and~\cite[Lemma~2.1]{Wool3}. 

Let $a_n$ be a sequence and for each $t\ge 1$ denote   
$$
A(t)=\sum_{1\le n\le t} a_n.
$$
Let $\psi: [1, N]\rightarrow \R$ be a differential function. Then 
$$
\sum_{n=1}^{N} a_n \psi(n)=A(N)\psi(N)-\int_{1}^{N} A(t)\psi'(t)dt.
$$

\begin{lemma} 
\label{lem:cont-gen} 
Suppose that  $\omega(T) \in \Z[T]$ is of degree $\deg \omega = k$.
Let $0<\alpha<1$. Let  $\varepsilon>0$ and 
$$
0<\zeta_1 \le N^{\alpha-2-\eps},   \quad  0<\zeta_2 \le N^{\alpha-k-1-\eps}.
$$
If  $W_{\omega}(x, y; N)\ge N^{\alpha}$ for  some $(x, y) \in \cR (u, v, \bzeta)$,
then for any $(a, b)\in \cR (u, v, \bzeta)$ we obtain 
$$
W_{\omega}(a, b; N)\ge  N^{\alpha}/2,
$$
provided  that $N$ is large enough. 
\end{lemma}

\begin{proof}
As in~\cite[Section~2.3]{ChSh-IMRN} we observe that for any $N$ there exists a sequence 
of complex numbers  $b_N(n)$ such that 
$$
b_N(n) \ll \log N, \qquad n=1, \ldots, N, 
$$ 
and $W_{\omega}(x, y; N)$ can be written as 
\begin{equation}
\label{eq:W}
W_{\omega}(x,y; N) =\sum_{n=1}^N  b_N(n)     \e\(xn+y \omega(n)\) .
\end{equation} 
For $\delta_1, \delta_2\in \R$ applying partial summation, we obtain 
\begin{equation} \label{eq:difference}
\begin{aligned}
W_{\omega}&(x+\delta_1,y+\delta_2; N) -W_{\omega}(x,y; N) \\
&=\sum_{n=1}^{N}b_N(n)\e(xn+y\omega(n))(\e(n\delta_1 +\omega(n)\delta_2)-1)\\
&=A(N)\psi(N)-\int_{1}^N A(t) \psi'(t)dt,
\end{aligned}
\end{equation}
where 
$$
A(t)=\sum_{n\le t} b_N(n)\e(xn+y\omega(n))
\quad \text{and} \quad  
\psi(t)=\e(t\delta_1+\omega(t) \delta_2)-1.
$$
Observe that   $A(N)\le N^{1+o(1)}$ as $N\rightarrow \infty$. Since $|\e(t)-1|\ll |t|$ for all $t\in \R$ and $\omega(N)\ll N^k$ for all large enough $N$,  we obtain 
$$
\psi(N)\ll N |\delta_1|+N^k |\delta_2|.
$$
Thus we derive 
\begin{equation}
\label{eq:first}
A(N)\psi(N)\ll N^{2+o(1)} |\delta_1|+ N^{k+1+o(1)} |\delta_2|.
\end{equation}
Furthermore, since $\psi'(t)\ll \delta_1+ t^{k-1} \delta_2$, we have 
\begin{equation}
\label{eq:second}
\begin{aligned}
\int_{1}^N A(t) \psi'(t)dt &\ll N^{1+o(1)} \int_{1}^{N} |\psi'(t)|dt\\
&\le N^{2+o(1)} |\delta_1|+ N^{k+1+o(1)} |\delta_2|.
\end{aligned}
\end{equation}
Combining~\eqref{eq:difference}  with~\eqref{eq:first} and~\eqref{eq:second} we obtain 
\begin{align*}
W_{\omega}(x+\delta_1,y+\delta_2; N) & -W_{\omega}(x,y; N)\\
&\ll  N^{2+o(1)} |\delta_1|+ N^{k+1+o(1)} |\delta_2|.
\end{align*}
Therefore, we conclude that for any fixed $\varepsilon>0$  and the choice of $\zeta_1$ and $\zeta_2$, the claim holds for all large enough $N$.  \end{proof}


\subsection{Large values of Weyls sums}
\label{sec: square count}

Let $0<\alpha<1$ and let  $\eps$  be  sufficiently small. We set 
\begin{equation}
\label{eq:zetaj}
\zeta_1=1/ \rf{N^{2+\eps-\alpha}}, \quad \zeta_2=1/ \rf{N^{k+1+\eps-\alpha}},
\end{equation}
and  divide $\T_2$ into $(\zeta_1\zeta_2)^{-2}$ 
squares of the form 
$$
[\ell \zeta_1, (\ell +1)\zeta_1)\times [m \zeta_2, (m+1)\zeta_2),
$$ 
where $\ell=0, \ldots, \zeta_1^{-1} -1$ and $m=0, \ldots, \zeta_2^{-1} -1$. Let $\fR$ be the collection of these squares. We now consider  
the subset of $\fR$ that  consists of squares which contain a large sum $W_{\omega}(\vu; N)$ for some $\vu = (x,y) \in \T_2$. More precisely, we denote
\begin{equation}
\label{eq:tR}
\widetilde \fR=\{\cR \in \fR:~\exists\, \vu \in \cR \text{ with } W_{\omega}(\vu; N)\ge N^{\alpha}\}.
\end{equation}

To present our results in full generality we assume that  there are positive  $s$ and $t$ such that  \begin{equation} 
\label{eq:admis}
 \int_0^1 \int_0^1 W_{\omega}(u,v; N)^{2 s } \, du\, dv\le N^{2s - t +o(1)}
\end{equation}
for $N \to \infty$. Then we specialise $s$ and $t$ to get concrete estimates.

\begin{lemma}
\label{lem:counting}
Suppose~\eqref{eq:admis} holds.   Then 
$$
\# \widetilde \fR \le (\zeta_1\zeta_2)^ {-1}  N^{2 s(1-\alpha)- t+o(1)}.
$$
\end{lemma}

\begin{proof} 
For each  $\cR\in \widetilde \fR$,   by  Lemma~\ref{lem:cont-gen}  we have $W_{\omega}(x, y; N)\ge N^{\alpha}/2$ for all $(x, y)\in \cR$. Combining this with~\eqref{eq:admis} we obtain 
$$
N^{ 2s \alpha}   \zeta_1\zeta_2 \# \widetilde \fR\ll \int_0^1 \int_0^1 W_{\omega}(u,v; N)^{2 s }  \, du\, dv\le N^{2s - t+o(1)},
$$
which yields  the desired bound.
\end{proof}

 \section{Proofs of main results}

\subsection{Proof of Theorem~\ref{thm:general}} 
We start with the following  statement which could be of independent interest.

\subsubsection{Lebesgue measure of large weighted Weyl sums} 
We continue to use $\lambda$ to denote the Lebesgue measure.

\begin{lemma}
\label{lem:bad set} Suppose~\eqref{eq:admis} holds.  
Let $f:\T_2\rightarrow \R$ be a $\rho$-H\"older function.   For $0<\alpha<1$ we have 
\begin{align*}
\lambda &  \(\{ z\in \R:~ \sup_{\vu \in f^{-1}(z)}   |W_{\omega}(\vu; N)|\ge N^{\alpha} \}   \) \\
 & \qquad \qquad \qquad  \le N^{(2-\alpha)(1-\rho)+(k+1-\alpha) +2s(1-\alpha)-t+o(1)}.
\end{align*}
\end{lemma}

\begin{proof}   We fix some sufficiently small   $\eps>0$ and define  the set 
$$
\fU = \bigcup_{\cR\in \widetilde \fR} \cR. 
$$
For $\cA\subseteq \T_2$ denote $f(\cA)=\{f(\vu): \vu\in \cA\}$. Observe that 
\begin{equation}\label{eq:cover-holder}
\left\{ z\in \R:~ \sup_{\vu \in f^{-1}(z)}  |W_{\omega}(\vu; N)|\ge N^{\alpha} \right\}    \subseteq f\( \fU \)\subseteq \bigcup_{\cR\in \widetilde \fR} f(\cR), 
\end{equation}
where $ \widetilde \fR$ is as in~\eqref{eq:tR}.

Since $f$ is  $\rho$-H\"older, for any $\cA\subseteq \T_2$ we obtain 
$$
\lambda\(f(\cA)\)\ll \( \diam\cA\)^{\rho}, 
$$
where $\diam \cA=\sup\{\|\va-\vb\|:~\va, \vb \in \cA\}$.  For each $\cR\in \widetilde\fR$, by~\eqref{eq:zetaj} we have 
$$
\diam \cR\ll \zeta_1\ll N^{\alpha-2-\varepsilon}.
$$

Combining with Lemma~\ref{lem:counting} and the estimate~\eqref{eq:cover-holder}, we derive 
\begin{align*}
\lambda & \(\{ z\in \R:~ \sup_{\vu \in f^{-1}(z)}   |W_{\omega}(\vu; N)|\ge N^{\alpha} \} \) \\ 
& \qquad \quad \ll  N^{(\alpha -2-\varepsilon)\rho}\# \widetilde \fR\  \ll N^{(2-\alpha+\varepsilon)(1-\rho)} N^{k+1-\alpha+\varepsilon} N^{2s(1-\alpha)-t}.
\end{align*}
Since $\varepsilon>0$ is arbitrary, 
this finishes the proof.
\end{proof}

\subsubsection{Conditional estimate}  
We  process a  similarly to as in~\cite[Section~4.1]{ChSh-IMRN}.

\begin{lemma}
\label{lem:conditional-estimate}
Suppose~\eqref{eq:admis} holds. Let $f:\T_2\rightarrow \R$ be a  $\rho$-H\"older function. Then 
\begin{equation}
\label{eq:Omegak-Cond}
\Omega_{k, \rho}\le  1- \frac{t-k -1+\rho}{2s+2-\rho}.
\end{equation} 
\end{lemma}
\begin{proof}
We  fix some  $\alpha > 1/2 $  and set
$$
N_i =    2^i, \qquad i =1, 2, \ldots.
$$

We now consider the set 
$$
\cB_{i} = \left\{z \in \R:~\exists\, \vu\in f^{-1}(z)    \text{ with }   W_{\omega}(\vu; N_i)\  \ge N_i^{\alpha}  \right\}. 
$$
By Lemma~\ref{lem:bad set} we have  
$$
\lambda\(\cB_i\)   \le N_i^{(2-\alpha)(1-\rho)+2 s(1-\alpha)+ k+1 - t-\alpha  +2\varepsilon+ o(1)}.
$$
We ask that  the parameters satisfy the following convergency condition 
$$
\sum_{i =1}^\infty  N_i^{(2-\alpha)(1-\rho)+2 s(1-\alpha)+ k+1 - t-\alpha +2\varepsilon + o(1)} < \infty, 
$$
which, due to the exponential growth of $N_i$ and the arbitrary small choice of $\varepsilon>0$, is equivalent to the inequality
\begin{equation}
\label{eq:condition-general}
(2-\alpha)(1-\rho)+2 s(1-\alpha)+ k+1 - t-\alpha <0.
\end{equation}
In this case, by the {\it Borel--Cantelli lemma\/}, we obtain that  
$$
\lambda \left(\bigcap_{q=1}^{\infty}\bigcup_{i=q}^{\infty} \cB_i \right)=0.
$$
Since
\begin{align*}
\{z \in \R:~\exists\, \vu\in f^{-1}(z)   & \text{ with }   W_{\omega}(\vu; N_i)\  \ge N_i^{\alpha} \}  \\
&\text{ for infinite many } i\in \N \} \subseteq \bigcap_{j=1}^{\infty}\bigcup_{i=j}^{\infty} \cB_i,
\end{align*}
we conclude that for almost all $z \in \R$ there exists $i_{z}$ such that for any $i\ge i_{z}$  one has  
\begin{equation}
\label{eq:y}
\sup_{\vu \in f^{-1}(z)}W_{\omega}(\vu; N_i)\le N_i^{\alpha}. 
\end{equation}

We fix  one of such $z\in \R$ in the following argument. For any $N\ge N_{i_{z}}$  we find  $i> i_z$ such that 
$$
N_{i-1}\le N< N_{i}.
$$
By Lemma~\ref{lem:control} and~\eqref{eq:y} we have 
$$
\sup_{\vu \in \T_2} |S_{\omega}(\vu; N)|  \ll\sup_{\vu \in \T_2}W_{\omega}(\vu; N_i) \ll N_i^{\alpha} \le N^\alpha.
$$

Note that the condition~\eqref{eq:condition-general} 
can be written as
$$
\alpha>\frac{2(1-\rho)+2s+k+1-t }{2s+2+1-\rho} = 1 - \frac{t-k-1+\rho }{2s+2-\rho},
$$
which finishes the proof. 
\end{proof}


\subsubsection{Concluding the proof} 
Similar to  the proof of Lemma~\ref{lem:cont-gen}, we write $W_{\omega}(u,v; N)$  as in~\eqref{eq:W}. 
Note that $b_N(n)\ll \log N$ for all $n=1, \ldots, N$.  Combining~\eqref{eq:W}  with Lemma~\ref{lem:IsN}  we see  that~\eqref{eq:admis} 
holds with 
\begin{equation}
\label{eq:st}
s=s_0(k) \mand t = k+1
\end{equation} 
where $s_0(k)$ is given by~\eqref{eq:s0 small k} and~\eqref{eq:s0 large k}, 
which after substitution in~\eqref{eq:Omegak-Cond} 
implies Theorem~\ref{thm:general}.

%
%
%

\subsection{Proof of Theorem~\ref{thm:v}}

It is sufficient to obtain the following analogue of the  estimate of Lemma~\ref{lem:bad set}, and then use the similar argument as in the proof of Theorem~\ref{thm:general}.

\begin{lemma}
\label{lem:bad set-y} Suppose~\eqref{eq:admis} holds.  
Let $f:\T_2\rightarrow \T$ with $f(x, y)=y$.   For $0<\alpha<1$ we have 
$$
\lambda  (\{ z\in \R:~ \sup_{\vu \in f^{-1}(y)}   |W_{\omega}(\vu; N)|\ge N^{\alpha} \}   )   \le N^{2s(1-\alpha)-t+2-\alpha+o(1)}.
$$
\end{lemma}  

\begin{proof} Let  $\pi_y: \T_2\rightarrow \T$ be the projection that  $\pi (x, y)=y$. Then 
\begin{equation}
\label{eq:ccover}
\{ z\in \T:~ \sup_{\vu \in f^{-1}(z)}   |W_{\omega}(\vu; N)|\ge N^{\alpha} \}\subseteq \pi_y (\bigcup_{\cR\in \widetilde \fR} \cR).
\end{equation}

Observe that each  square of $\cR$ is of the following form  
$$
[\ell \zeta_1, (\ell +1)\zeta_1)\times [m \zeta_2, (m+1)\zeta_2)
$$ 
for some $\ell$ and $m$, see~\eqref{eq:zetaj}. Combining with~\eqref{eq:cover-holder} and~\eqref{eq:ccover}, we derive   
\begin{align*}
\lambda  &  \(\left\{ z\in \T:~ \sup_{\vu \in f^{-1}(z)}   |W_{\omega}(\vu; N)|\ge N^{\alpha} \right \}\)\\
&   \qquad \qquad \ll   \lambda \(\pi_y \(\bigcup_{\cR\in \widetilde \fR} \cR\)\) \ll  \zeta_2 \#\widetilde \fR  \ll  N^{2s(1-\alpha)-t+2-\alpha +o(1)}. 
\end{align*}
Applying  Lemma~\ref{lem:counting} 
we   finish the proof.
\end{proof}

Applying Lemma~\ref{lem:bad set-y}  with $N = 2^i$, $i =0,1, \ldots$, we obtain the desired results provided 
$$
2s(1-\alpha)-t+2-\alpha< 0
$$
or, equivalently 
$$
\alpha > 1 -  \frac{t-1}{2s +1}.
$$
With the choice~\eqref{eq:st} we conclude the proof.

\subsection{Proof of Theorem~\ref{thm:c}}

Recall that we fix $0<r<1$. It is sufficient to prove the following estimate.

\begin{lemma}
\label{lem:large-c} Suppose~\eqref{eq:admis} holds.  
Fix $0<r<1$.  The for $0<\alpha<1$ we have 
$$
\lambda   \(\{ \vz\in \T_2:~ \sup_{\vu \in \cC(\vz,r)}   |W_{\omega}(\vu; N)|\ge N^{\alpha} \}   \)  \le N^{k+1-\alpha+ 2s(1-\alpha)-t+o(1)}
$$
\end{lemma} 

\begin{proof}
For each  square $\cR$ as in Section~\ref{sec: square count}, denote 
$$
\cC(\cR)=\bigcup_{\vz\in \cR} \cC(\vz, r).
$$
Observe that 
$$
\{ \vz\in \T_2:~ \sup_{\vu \in\cC(\vz,r)}  |W_{\omega}(\vu; N)|\ge N^{\alpha} \}   
\subseteq \bigcup_{\cR\in \widetilde \fR} \cC(\cR).
$$
Moreover, for each $\cR\in \widetilde \fR$ we have 
$$
\lambda(\cC(\cR))\ll \zeta_1.
$$
Applying Lemma~\ref{lem:counting}  we derive 
\begin{align*}
\lambda&\(\{ \vz\in \T_2:~ \sup_{\vu \in\cC(\vz,r)}   |W_{\omega}(\vu; N)|\ge N^{\alpha} \} \)\\
&   \qquad \qquad \qquad    \le  \zeta_1\# \widetilde \fR \ll N^{k+1-\alpha +2s(1-\alpha)-t},
\end{align*}
which gives the desired bound.
\end{proof}

Applying Lemma~\ref{lem:large-c}  with $N = 2^i$, $i =0,1, \ldots$, we obtain the desired results provided 
$$
k+1-\alpha+ 2s(1-\alpha)-t < 0
$$
or, equivalently 
$$
\alpha > 1 -  \frac{t-k}{2s +1}.
$$
With the choice~\eqref{eq:st} we conclude the proof.

\section{Comments}

We see from~\eqref{eq:Omegak-Cond} that  any reduction in the value of $s_0(k)$ in the 
condition on $s$ in Lemma~\ref{lem:IsN} immediately leads to an improvement
of  Theorems~\ref{thm:general}, \ref{thm:v} and Theorem~\ref{thm:c}, see also Remarks~\ref{rem:non-integr} for one of the possible 
ways to achieve this.   
 
Certainly, the case of non-polynomial functions $\omega(T)$, such as, for example,  $\omega(T) = T^\kappa$ with some $\kappa \in \R$, 
which has also been considered in~\cite{ErdSha}, are of interest as well. Our method can be applied to such functions as well, 
provided appropriate mean value theorems become available.  It  is easy to see that one can have analogues of Lemmas~\ref{lem:cont-gen} 
and~\ref{lem:bad set} for any  function $\omega: \N \to \R$ with  some smoothness conditions
such as 
\begin{equation}
\label{eq:deriv}
\omega'(x) \ll x^{\kappa-1+o(1)}, \qquad \text{as}\ x\to \infty.
\end{equation} 
We can also define natural analogues of the sums $S_{\omega}(x,y; N)$ and
$ W_{\omega}(x,y; N)$.  One then easily checks that our method produces 
nontrivial results for the sums $S_{\omega}(x,y; N)$ for any function $\omega$ 
satisfying~\eqref{eq:admis} and~\eqref{eq:deriv} with $t > \kappa$.

Our methods can also be used to address  the following general  scenario. Let $(\Gamma, \mu)$ be a measure space and suppose
that  for each $\gamma \in \Gamma$  there is a corresponding  set  $\cA_\gamma \subseteq\T_2$ satisfying
 some  ``regular" conditions.    Then the goal is there may exist some positive $\vartheta<1$, depending only on $\mu$ and $\Gamma$, 
 (and the properties of the sets $\cA_\gamma$)
 such that for $\mu$-almost all $\gamma\in \Gamma$ we have 
$$
\sup_{\vu\in \cA_{\gamma} } |S_{\omega}(\vu; N)|\le N^{\vartheta+o(1)}.
$$
For example,  the sets $\cA_\gamma$, $\gamma\in \Gamma$, in~\cite{ErdSha}  is a family of lines with rational direction, while
 the sets $\cA_\gamma$, $\gamma\in \Gamma$, in  Theorem~\ref{thm:general} is the level sets of some H\"older function $f$. Furthermore, the sets $\cA_\gamma$, $\gamma\in \Gamma$, of Theorem~\ref{thm:c}  is a family of circles of fixed radius $r$.
 Certainly more general sets are also of interest and can be investigated via our approach.

\section*{Acknowledgement}

The authors would like to thank Julia Brandes, Burak  Erdo{\v g}an, George Shakan and Trevor Wooley for helpful discussions and patient answering their questions.  In particular,  the authors are very grateful to Trevor Wooley for directing them to the results of~\cite[Section~14]{Wool5}
which have led  to improved bounds. 

This work was  supported   by ARC Grant~DP170100786.

\end{document}